\documentclass[a4paper,british,12pt]{amsart}

\usepackage[T1]{fontenc}
\usepackage{hyperref,lmodern,babel}

\usepackage[all]{xy}
\usepackage{stmaryrd}

\newtheorem{lem}{Lemma}[section]
\newtheorem{prop}[lem]{Proposition}
\newtheorem{thm}[lem]{Theorem}
\newtheorem{cor}[lem]{Corollary}
\newtheorem*{lem*}{Theorem}

\theoremstyle{definition}

\newtheorem*{ex*}{Example}

\newtheorem*{rem*}{Remark}

\newcommand{\norm}[1]{\left\Vert#1\right\Vert}
\newcommand{\abs}[1]{\left\vert#1\right\vert}

\newcommand{\set}[1]{\left\{#1\right\}}

\newcommand{\appli}[3]{#1 \, : \, #2 \To #3}
\newcommand{\operator}[5]{ 
\begin{eqnarray*}
#1 :  \quad #2 &\To& #3 \\ 
#4 &\mapsto& #5 
\end{eqnarray*}}
\newcommand{\rl}{\mathbb R}
\newcommand{\cx}{\mathbb C}
\newcommand{\ir}{\mathbb Z}

\newcommand{\sph}{\mathbb S}

\newcommand{\tor}{\mathbb T}
\newcommand{\To}{\longrightarrow}
\newcommand{\OO}{\mathcal O}
\newcommand{\bP}{\mathbb{P}}

\newcommand{\setC}{\cx}
\newcommand{\setZ}{\ir}
\newcommand{\setR}{\rl}
\newcommand{\bT}{\tor}

\DeclareMathOperator{\ric}{Ric}

\DeclareMathOperator{\inj}{inj}
\DeclareMathOperator{\supp}{supp}

\DeclareMathOperator{\re}{Re} 
\DeclareMathOperator{\im}{Im}

\begin{document}

\author{Olivier Biquard and Vincent Minerbe}
\address{UPMC Universit\'e Paris 6, UMR 7586, Institut de Math\'ematiques
  de Jussieu}
\title{A K\"ummer construction for gravitational instantons}

\begin{abstract}
  We give a simple and uniform construction of essentially all known deformation
  classes of gravitational instantons with ALF, ALG or ALH asymptotics
  and nonzero injectivity radius. We also construct new ALH Ricci flat
  K\"ahler metrics asymptotic to the product of a real line with a flat
  3-manifold.
\end{abstract}

\maketitle

\section*{Introduction}

The aim of this paper is to give a direct and uniform construction for several 
Ricci-flat K\"ahler four-manifolds with prescribed asymptotics, `ALF', `ALG' or `ALH'.
This basically means that these complete Riemannian manifolds have only one end, which is diffeomorphic, 
up to a finite covering, to the total space of a $\tor^{4-m}$-fibration $\pi$ over $\rl^m$ minus a ball
and carries a metric that is asymptotically adapted to this fibration in the following sense. When the fibration 
at infinity is trivial, the metric merely goes to a flat metric on $\rl^m \times \tor^{4-m}$ (without holonomy). 
When $m=3$, the fibration at infinity may be non-trivial and in this case, the metric goes to 
$\pi^* g_{\rl^3} + \eta^2$ where $\eta$ is a connection one-form on the $\sph^1$-fibration, up to scaling (cf. 
\cite{Mi2} for details). The metric is called ALF when the `dimension at infinity' $m$ is $3$, ALG when $m=2$, ALH 
when $m=1$. These asymptotics are generalizations of the familiar ALE (Asymptotically Locally Euclidean) case, for 
which the model at infinity is the Euclidean four-space (in a nutshell: $m=4$). 

Most of our examples of Ricci-flat K\"ahler four-manifolds are simply-connected, hence hyperk\"ahler
(i.e. with holonomy inside $SU(2)$), providing examples of gravitational instantons, namely non-compact hyperk\"ahler 
four-manifolds with decaying curvature at infinity. These manifolds are of special interest in quantum gravity and 
string theory, hence some motivation to understand examples. Previous constructions of gravitational instantons 
were either explicit \cite{EH,GH}, based on hyperk\"ahler reduction \cite{Kro,Dan}, 
gauge theory \cite{AH,CK1,CK2}, twistorial methods \cite{Hit,CH} or on the Monge-Amp\`ere method of Tian and Yau 
\cite{TY1,TY2}, see also \cite{Joy,Kov,Hei,San}.

The technique we advertise here is inspired from the 
famous K\"ummer construction for $K3$ surfaces: starting from the torus $\tor^4=\cx^2/\ir^4$,
we may consider the complex orbifold $\tor^4/\pm$, which has $16$ singularities isomorphic to 
$\cx^2/\pm$; blowing them up, we obtain a $K3$ surface. 
The physicist D. Page \cite{Pag} noticed that this point of view makes it possible to grasp 
some idea of what the Ricci flat K\"ahler metric provided by Yau theorem on the $K3$ surface 
looks like. The recipe is the following. The desingularization of $\cx^2/\pm$ carries an 
explicit Ricci-flat K\"ahler metric known as the Eguchi-Hanson metric \cite{EH}. So the Ricci
flat metric on the $K3$ surface should look like this Eguchi-Hanson metric near each exceptional 
divisor and resemble the flat metric (issued from $\tor^4$) away from them. Such an idea has been 
carried out rigorously by twistorial methods in \cite{LS,Top}. N. Hitchin \cite{Hit} also pointed
out a twistorial argument leading to a Ricci-flat metric on the (non-compact) minimal resolution
of $\set{x^2-z y^2 = z} \subset \cx^3$.

In this short paper, we carry out a rather elementary deformation argument providing Ricci-flat K\"ahler
metrics on the minimal resolution of numerous Ricci-flat K\"ahler orbifolds. We can typically start from 
$\rl^m \times \tor^{4-m} /\pm$ for $m=1,2,3$, $(\rl^2 \times \tor^2) / \ir_k$ for $k=3,4,6$ and 
$(\rl \times \tor^3)/(\ir_2 \times \ir_2)$. Details on these quotients are given in the text. 
The construction also works for some non-flat Ricci-flat K\"ahler orbifolds and we will use it 
on (singular) quotients of the Taub-NUT gravitational instanton \cite{Haw,Leb}. Let $(X,\omega_0)$ 
denote any of these complex K\"ahler orbifolds and $\hat{X}$ be its minimal resolution. We denote the 
exceptional divisors by $E_1, \dots, E_p$ and the Poincar\'e dual of $E_j$ by $PD[E_j]$.

\begin{thm}
Let $a_1, \dots, a_p$ denote some positive parameters. Then for every small enough positive number $\epsilon$, there 
is a Ricci-flat K\"ahler form $\omega$ on $\hat{X}$ in the cohomology 
class $[\omega_0]  - \epsilon \sum a_j PD[E_j]$; 
it is moreover asymptotic to the initial metric $\omega_0$ on  $X$.
\end{thm}

The proof of the theorem relies on a simple gluing procedure and is
essentially self-contained. Even though most of the metrics we build
have already been constructed in the literature by other ways
(cf. above), we believe that our construction is interesting, because
it is very simple and gives a very good approximation coming from the
desingularization procedure. More precisely, the construction
described in this paper provides a way to build members of nearly all
known deformation families of ALF gravitational instantons, by
starting from the explicit Taub-NUT metric and Kronheimer's ALE
gravitational instantons \cite{Kro}. The `nearly' accounts for the
Atiyah-Hitchin metric \cite{AH}, which seems to play a special
role. And more generally, apart from the Atiyah-Hitchin metric, we
have the striking fact that this construction yields members of all
known deformation families of gravitational instantons with positive
injectivity radius. So we somehow get a global and concrete
understanding of all these families.

Let us now overview the Ricci-flat manifolds that we obtain with this
technique.  When $X= (\rl^2 \times \tor^2) / \ir_k$ for $k=2,3,4,6$, we
obtain ALG gravitational instantons. For $X_1=\rl \times \tor^3 /\pm$, our
construction yields an ALH gravitational instanton. These ALG and ALH
examples have been constructed recently in a general Tian-Yau
framework on rational elliptic surfaces in \cite{Hei} (see also
\cite{Kov,San}); in \cite{Hei}, they correspond to isotrivial elliptic
fibrations. Two other quotients, $X_2= (\setR \times \tor^3)/(\setZ_2 \times \setZ_2)$ and
$X_{2,2}=(\setR \times \tor^3)/(\setZ_2\times\setZ_2\times\setZ_2)$, provide examples of ALH
Ricci-flat metrics asymptotic to $\rl_+ \times F_2$ (resp. $\rl_+ \times
F_{2,2}$), where $F_2$ is the compact flat orientable three-manifold
arising as a flat $\tor^2$ bundle over $\sph^1$ with monodromy
$\ir_2$, and $F_{2,2}$ is its $\setZ_2$ quotient $ \tor^3/\setZ_2\times\setZ_2$. These
examples are just global quotients of the first ALH space $\hat
X_1$. Note that, in this case, $\hat X_2$ is K\"ahler Ricci-flat but not
hyperk\"ahler, and $\hat X_{2,2}$ is only locally K\"ahler Ricci flat,
since the last involution is real. Later, in section
\ref{sec:other-alh-ricci}, we also construct actions leading to ALH
Ricci-flat K\"ahler metrics with one end asymptotic to $\setR_+\times F_3$, $\setR_+\times
F_4$ and $\setR_+\times F_6$, where $F_3$, $F_4$ and $F_6$ are the flat
3-manifolds with monodromy $\setZ_3$, $\setZ_4$ and $\setZ_6$. Summarizing this
means that we have constructed examples of ALH Ricci flat metrics,
with an end asymptotic to $\setR_+\times F$, for all possible flat 3-manifolds
$F$. All are K\"ahler except $\hat X_{2,2}$.

Now pass to ALF gravitational instantons. When $X= \rl^3 \times \sph^1 /\pm$,
we obtain a direct PDE construction of the Hitchin metric mentioned
above \cite{Hit}, complementing the twistorial initial description;
this is an ALF gravitational instanton of dihedral type (in the sense
of \cite{Mi2}). Next, for any binary dihedral group $D_k$ ($k>2$), we
can deal with $X= \cx^2/D_k$, endowed with the Taub-NUT metric
(details are given at the end of the text). The resulting Ricci-flat
metrics provide ALF gravitational instantons of dihedral type and they
probably coincide with the $D_k$ gravitational instantons of
\cite{CK1,CH}. The same procedure using an action of a cyclic group
instead of $D_k$ would lead to the ALF gravitational instantons of
cyclic type that is multi-Taub-NUT metrics, given by the
Gibbons-Hawking ansatz (\cite{Mi2}).

Finally, any ALF gravitational instanton has an end which is
asymptotic to $(A,+\infty)\times S$, with $S$ a quotient of $\sph^3$ by a cyclic or
binary dihedral group, or $\sph^2\times \sph^1$ or $\sph^2\times \sph^1/\pm$. We prove a formula
(theorem \ref{thm:euler-eta}) relating the Euler characteristic to the
adiabatic $\eta$-invariant of the boundary $S$. The possible topologies
of $S$ are well-known, but the possible orientations induced on $S$
are less clear (they correspond to `positive mass' or `negative
mass'). Since the $\eta$-invariant detects the orientation, we deduce
that for the dihedral ALF gravitational instantons, the only
boundaries for which both orientations are possible occur for the
groups $D_4$ and $D_3$: the reverse orientations are realized by the
$D_0$ and $D_1$ ALF gravitational instantons, which conventionally are
the Atiyah-Hitchin metric and its double cover, both having negative
mass. Also note that the $D_2$ ALF gravitational instanton is the
Hitchin metric desingularizing $\rl^3 \times \sph^1 /\pm$, which has zero
mass (one way to understand is to say that the flat space $\rl^3 \times
\sph^1 /\pm$ is hyperk\"ahler for both orientations).

The paper is organized as follows. In the first section, we have
chosen to give a detailed proof in the simplest case, that is $X=\rl \times
\tor^3 /\pm$. The necessary adaptations for the other cases are
explained in the second section. The third section contains the
proof of the formula on the $\eta$ invariant, and in the last one, we give
a construction of the other ALH Ricci-flat K\"ahler surfaces. Finally,
an appendix briefly reviews some facts about the weighted analysis on
ALF, ALG or ALH manifolds, needed for the construction.

\noindent\emph{Acknowledgments.} The authors thank Sergei Cherkis for
useful discussions, and Hans-Joachim Hein for carefully checking the
paper and suggesting proposition \ref{prop:alh22}.

\section{A K\"ummer construction}

In this section, we carefully build an ALH gravitational instanton asymptotic 
to $\rl \times \tor^3$. First, we consider the quotient $X := \rl \times \tor^3 /\pm$.
This is a complex orbifold with eight singular points, corresponding to the fixed points 
of minus identity on $\tor^3 = \rl^3/\ir^3$. These are 
all rational double points and we may blow them up to get a non-singular complex manifold
$\hat{X}$. We will build an \emph{approximately} Ricci-flat metric on $\hat{X}$ by patching 
together two type of metrics : the flat metric away from the exceptional divisors and the 
Eguchi-Hanson metric near the divisors.

\subsection{The approximately Ricci-flat K\"ahler metric}

The Eguchi-Hanson metric can be described as follows. Take $\cx^2/\pm$ and blow up the origin
to get the minimal resolution $\appli{\pi}{T^*\cx P^1}{\cx^2/\pm}$. Outside the zero section $\cx P^1= \pi^{-1}(0)$ 
in $T^*\cx P^1$, this map $\pi$ is a biholomorphism. Let $z=(z_1,z_2)$ denote the standard complex coordinates
on $\cx^2$. The formula
$$
\phi_{EH}(z_1,z_2) := \frac12 \left( \sqrt{1+ \abs{z}^4} + 2 \log \abs{z} - \log \left( 1 + \sqrt{1+ \abs{z}^4} \right) \right),
$$
defines a function on $\cx^2/\pm$ and therefore on $T^*\cx P^1 \backslash \pi^{-1}(0)$. Then 
$\omega_{EH} := d d^c \phi_{EH}$ extends on the whole $T^*\cx P^1$ as a K\"ahler form, which turns out to be
Ricci-flat. Moreover, the K\"ahler form $\omega_{EH}$ is asymptotic to the flat K\"ahler form 
$\omega_{EH,0} = d d^c \phi_{EH,0}$, with $\phi_{EH,0}(z)=\frac{\abs{z}^2}{2}$:
$$
\nabla^k (\phi_{EH} - \phi_{EH,0}) =\OO(\abs{z}^{-2-k}) \quad \text{and} \quad
\nabla^k (\omega_{EH} - \omega_{EH,0}) =\OO(\abs{z}^{-4-k}).
$$
For future reference, let us point that the Eguchi-Hanson metric admits a parallel symplectic $(2,0)$ form, 
extending (the pull-back of) $dz_1 \land dz_2$.

Now pick a smooth non-increasing function $\chi$ on $\rl_+$ that is identically $1$ on $[0,1]$ 
and vanishes on $[2,+\infty)$. Given a small positive number $\epsilon$, we introduce the cut-off function 
$\chi_\epsilon (z) := \chi(\sqrt{\epsilon} \abs{z})$, on $T^*\cx P^1$, and we define 
\begin{equation}\label{patch}
\phi_{EH,\epsilon} := \chi_\epsilon \phi_{EH} + (1-\chi_\epsilon) \phi_{EH,0} .
\end{equation}
Then $\omega_{EH,\epsilon} := dd^c \phi_{EH,\epsilon}$ is a $(1,1)$-form on $T^*\cx P^1$ which 
coincides with the Eguchi-Hanson K\"ahler form for $\abs{z} \leq \frac{1}{\sqrt{\epsilon}}$ and 
coincides with the flat K\"ahler form on $\cx^2 /\pm$ for $\abs{z} \geq \frac{2}{\sqrt{\epsilon}}$.
In between, we have the following controls:
$$
\abs{\nabla^k (\phi_{EH,\epsilon} - \phi_{EH,0})} \leq c(k) \sqrt{\epsilon}^{2+k}
\quad \text{and} \quad
\abs{\nabla^k (\omega_{EH,\epsilon} - \omega_{EH,0})} \leq c(k) \sqrt{\epsilon}^{4+k}.
$$
In particular, for small $\epsilon$, $\omega_{EH,\epsilon}$ is a K\"ahler form on the whole $T^*\cx P^1$.
(The letter $c$ will always denote a positive constant whose value changes from line to line
and we sometimes write $c(\dots)$ to insist on the dependence upon some parameters.)

We can now describe precisely our approximately Ricci flat K\"ahler metric $\omega_\epsilon$ on $\hat{X}$. 
Let $\rho$ denote the distance to the singular points in the flat orbifold $\rl \times \tor^3/\pm$. The function $\rho$ 
can also be seen as the $\omega_0$-distance to the exceptional divisors in $\hat{X}$. We will denote the 
connected components of $N := \set{\rho \leq \frac18}$ by $N_{j}$, $1 \leq j \leq 8$, and the remaining 
part of $\hat{X}$ by $W$. Each $N_j$ contains an exceptional divisor $E_j$ and $N_j \backslash E_j$ 
is naturally a punctured ball of radius $\frac18$ and centered in $0$ in 
$\cx^2/\pm \cong T^* \cx P^1 \backslash \cx P^1$. Thanks to a $\frac{1}{\epsilon}$-dilation, we may 
therefore identify $N_j$ with the set $N(\epsilon):=\set{\abs{z} \leq \frac1{8 \epsilon}}$
in $T^* \cx P^1$ and then define $\omega_\epsilon$ on $N_j$ by 
$\omega_\epsilon := \epsilon^2 \omega_{EH,\epsilon}$. On $W$, we then let $\omega_\epsilon$ coincide 
with the flat $\omega_0$. Owing to the shape of $\omega_{EH,\epsilon}$, this defines a smooth K\"ahler metric 
on the whole $\hat{X}$. By construction, $\omega_\epsilon$ is the flat $\omega_0$ for $\rho \geq 2\sqrt{\epsilon}$,
a (scaled) Eguchi-Hanson metric for $\rho \leq \sqrt{\epsilon}$ and obeys 
\begin{equation}\label{estimomega}
\abs{\nabla^k (\omega_{\epsilon} - \omega_{0})} \leq c(k) \epsilon^{2-\frac{k}{2}}
\end{equation}
between these two areas (beware the scaling induces a nasty $\epsilon^{-k-2}$ factor). The Ricci form $\ric_\epsilon$ 
vanishes outside the domain $\set{\sqrt{\epsilon} \leq \rho \leq 2 \sqrt{\epsilon}}$ 
and obeys $\abs{\ric_{\epsilon}} \leq c \epsilon$: it is small. We now wish to deform this approximately Ricci-flat K\"ahler metric into a genuine Ricci-flat K\"ahler metric.

\subsection{The nonlinear equation}

In view of building a Ricci-flat K\"ahler form $\omega = \omega_\epsilon + d d^c \phi$ in the K\"ahler class of $\omega_\epsilon$, we wish to solve  the complex Monge-Amp\`ere equation 
\begin{equation}\label{mongeampere}
\left( \omega_\epsilon + dd^c \psi \right)^2 = e^{f_\epsilon} \omega_\epsilon \land \omega_\epsilon,
\end{equation}
where $f_\epsilon$ is essentially a potential for the Ricci form $\ric_\epsilon$ of $\omega_\epsilon$ : $\ric_\epsilon = \frac12 dd^c f_\epsilon$. This is the classical approach of Aubin-Calabi-Yau (cf. \cite{Joy} for instance). Our framework gives an explicit function $f_\epsilon$, which we can describe as follows. Let $\zeta_1,\zeta_2$ denote the complex coordinates on $\cx^2$ and 
consider the $(2,0)$ form $d\zeta_1 \land d\zeta_2$ on $\rl \times \tor^3 = \cx^2 /\ir^3$. It is still defined on 
$X=\rl \times \tor^3/\pm$ and then lifts into a holomorphic $(2,0)$ form $\Omega$ on $\hat{X}$. Since $\hat{X}$
is a crepant resolution of $X$, this $(2,0)$ form $\Omega$ does not vanish along the exceptional divisors, providing a genuine symplectic $(2,0)$ form. We can then choose the following function $f_\epsilon$:
\begin{equation}\label{potricci}
f_\epsilon := \log \left( \frac{\Omega \land \bar{\Omega} }{ \omega_\epsilon \land \omega_\epsilon } \right).
\end{equation}
In other words, the right-hand side of (\ref{mongeampere}) is simply $\Omega \land \bar{\Omega}$.
Observe $f_\epsilon$ is compactly supported inside $\set{\sqrt{\epsilon} \leq \rho \leq 2 \sqrt{\epsilon}}$ and obeys:
\begin{equation}\label{estimpotricci}
\abs{\nabla^k f_\epsilon} \leq c(k) \epsilon^{2-\frac{k}{2}}.
\end{equation}

\subsection{The linear estimate}

The linearization of the Monge-Amp\`ere operator is essentially the Laplace operator $\Delta_\epsilon$. 
We need to show that it is an isomorphism between convenient Banach spaces and that its inverse 
is uniformly bounded for small $\epsilon$. Let us introduce the relevant functional spaces. We denote 
the $\rl$-variable in $\rl \times \tor^3$ by $t$ and let $r:= \abs{t}$. We may choose a smooth positive 
function $r_\epsilon$ with the following properties: 
\begin{itemize}
 \item it is equal to $r$ wherever $r$ is larger than $1$;
 \item it coincides with the distance $\rho$ to the exceptional divisors wherever $2\epsilon \leq \rho \leq \frac18$;
 \item it is identically $\epsilon$ wherever $\rho \leq \epsilon$;
 \item it is a non-decreasing function of $\rho$ in the domain where $\epsilon \leq \rho \leq 2\epsilon$;
 \item it remains bounded between $\frac18$ and $1$ on the part of $\hat{X}$ where $\rho > \frac18$ and $r < 1$. 
\end{itemize}
Let $k$ be a nonnegative integer and $\alpha$ be a number in $(0,1)$. Given positive real numbers $a$ and $b$,
we let $w_{\epsilon,i}$ denote the continuous function which coincides with
$r_\epsilon^{a+i}$ on $W$ and $r_\epsilon^{b+i}$ on $N$. Then we define the Banach
space $C^{k,\alpha}_{\epsilon,a,b}$ as the set of $C^{k,\alpha}$ functions $u$ for
which the following quantity is finite:
\begin{multline*}
\norm{u}_{C^{k,\alpha}_{\epsilon,a,b}} :=  
\sum_{i=0}^k \sup \abs{w_{\epsilon,i} \nabla_\epsilon^i u}_\epsilon \\
+ \sup_{d_\epsilon(x,y) < \inj_\epsilon} \abs{  \min(w_{\epsilon,k+\alpha}(x),w_{\epsilon,k+\alpha}(y)) \frac{\nabla^k_\epsilon u(x) - \nabla^k_\epsilon u(y)}{d_\epsilon(x,y)^\alpha} }_\epsilon.
\end{multline*}
This formula provides the norm. The subscripts $\epsilon$ mean that everything is computed with respect to the metric $g_\epsilon$; in particular, $\inj_\epsilon$ denotes the \emph{positive} injectivity radius of $g_\epsilon$. In this definition, since $d_\epsilon(x,y) < \inj_\epsilon$, we can compare the values of $\nabla^k u_\epsilon$ at $x$ and $y$ through the parallel transport along the \emph{unique} minimizing geodesic from $x$ to $y$. 

On the exterior domain $W$, a $C^{k,\alpha}_{\epsilon,a,b}$ control on $u$ means it decays like $r^{-a}$ (with a 
corresponding natural estimate on the derivatives). As expected, the parameter $\epsilon$ only matters near the exceptional
divisors. Recall each $N_{j}$ is identified with the domain $N(\epsilon):=
\set{\abs{z}\leq \frac{1}{8\epsilon}}$ in $T^* \cx P^1$. Accordingly, we may carry any function $u$ on
$N_{j}$ to a function $u_\epsilon$ on $N(\epsilon)$. A $C^{0}_{\epsilon,a,b}$ control on $u$ basically means $\abs{u_\epsilon} \leq c \epsilon^{-b} \abs{z}^{-b}$ for $2 \leq \abs{z} \leq \frac{1}{8\epsilon}$ and $\abs{u_\epsilon} \leq c \epsilon^{-b}$ where $\abs{z} \leq 2$. Note that standard (scaled) Schauder estimates imply the following control:
\begin{equation}\label{schauder}
\norm{u}_{C^{k+2,\alpha}_{\epsilon,a,b}} \leq c(k,\epsilon) \; \left( \norm{u}_{C^{0}_{\epsilon,a,b}} 
+ \norm{\Delta_\epsilon u}_{C^{k,\alpha}_{\epsilon,a+2,b+2}}\right).
\end{equation}
In this asymptotically cylindrical setting, we may have used exponential weights instead of polynomial 
weights. We have made this choice because 1) these weights suffice for our purpose here and 2) 
they will be adapted to the other settings (ALG, ALF, cf. appendix).

The relevant Banach spaces for us will turn out to be 
$$
E^{a,b}_\epsilon := \rl \tilde{r} \oplus C^{2,\alpha}_{\epsilon, a,b} 
\quad \text{and} \quad F^{a,b}_\epsilon := C^{0,\alpha}_{\epsilon, a+2,b+2}. 
$$
The notation $\tilde{r}$ stands for the function $\phi \circ r$,
where $\phi$ is any smooth non-decreasing function on $\rl_+$ which 
is identically $1$ on $[0,1]$ and coincides with the identity on $[2,+\infty)$.
We fix the norm on $\rl \tilde{r}$ by setting $\norm{\tilde{r}}:=1$ and we endow 
$E^{a,b}_\epsilon$ with the sum of the norms of the two factors. Note we have dropped 
the dependence on $\alpha$, which will not play a major role. The Laplacian 
$\Delta_\epsilon$ defines a bounded operator between $E^{a,b}_\epsilon$ 
and $F^{a,b}_\epsilon$ and we need to invert it. 

\begin{lem}\label{surj}
The map $\appli{\Delta_\epsilon}{E^{a,b}_\epsilon}{F^{a,b}_\epsilon}$ is an isomorphism.
\end{lem}

\begin{proof}
Any function $u$ in $E^{a,b}_\epsilon$ has the following asymptotics: $u = \lambda r + \OO(r^{-a})$. 
When $u$ is harmonic, integration by part yields for large $R$:
$$
0= \int_{B_R} \Delta_\epsilon u = - \int_{\partial B_R} \partial_r u = - \lambda + \OO(R^{-a-1}),
$$
with $B_R = \set{r\leq R}$. So $\lambda=0$ and $u$ goes to zero at infinity, hence vanishes. This proves injectivity. Let us turn to surjectivity. 
Given $f \in F^{a,b}_\epsilon$, weighted analysis (cf. appendix) provides a solution $u$ to $\Delta_\epsilon u=f$, with 
$$
u = \lambda \tilde{r} + \eta + v
$$
for some constants $\lambda$ and $\eta$, and some $C^{2,\alpha}$ function $v$ bounded in $C^{2,\alpha}_{\epsilon,a,b}$ (this estimate follows from (\ref{estimG}) and (\ref{schauder})). Of course, the constant $\mu$ can be dropped (for there is only one end) and we have proved the surjectivity. Note the parameter $b$ is not relevant here. It enters the picture only when one seeks uniform bounds, which is the next topic we tackle. 
\end{proof}

\begin{lem}\label{estimlin}
Let $a$ and $b$ be positive numbers, with $b < 2$. Then there is a constant $c$ such 
that for every small $\epsilon$ and every $u \in E^{a,b}_\epsilon$,
$$
\norm{u}_{E^{a,b}_\epsilon} \leq c \norm{\Delta_\epsilon u}_{F^{a,b}_\epsilon}.
$$
\end{lem}

\begin{proof}
Assume this statement is false. Then there is a sequence of positive numbers $\epsilon_i$ and 
a sequence of functions $u_i$ such that: $(\epsilon_i)_i$ goes to zero, $\norm{u_i}_{E^{a,b}_{\epsilon_i}}=1$
and $\norm{\Delta_{\epsilon_i} u_i}_{F^{a,b}_{\epsilon_i}}$ goes to zero. Then $u_i = \lambda_i \tilde{r} + v_i$,
with $(\lambda_i)$ bounded in $\rl$ and $(v_i)$ uniformly bounded in $C^{2,\alpha}_{\epsilon_i, a,b}$. In particular, we may assume that $(\lambda_i)$ goes to some $\lambda_\infty$ and, with Arzela-Ascoli theorem, that $(v_i)$ converges to $v_\infty$ on compact subsets of $\hat{X}$ minus the exceptional divisors. Moreover, $u_\infty := \lambda_\infty \tilde{r} + v_\infty$ is a harmonic function on the regular part of $\rl \times \tor^3/\pm$ (with its flat metric) and is bounded by a constant times $\rho^{-b}$ near each singular point, which is less singular than the Green function on $\rl^4$, for $b<2$. It follows that $u_\infty$ can be lifted into a smooth harmonic function $\hat{u}_\infty$ on the whole $\rl \times \tor^3$. Moreover, $r^{a} v_\infty$ is bounded so $\hat{u}_\infty = \lambda_\infty r + \OO(r^{-a})$. Integration by part yields for large $R$:
$$
0= \int_{B_R} \Delta \hat{u}_\infty = - \int_{\partial B_R} \partial_r \hat{u}_\infty = -2 \lambda_\infty + \OO(R^{-a}).
$$
So $\lambda_\infty=0$ and $\hat{u}_\infty$ goes to zero at infinity, hence vanishes: $v_\infty=0$. It follows that $\norm{v_i}_{C^{0}(K)}$ goes to zero for every compact set $K$ outside the exceptional divisors. Using Lemma \ref{estimG}, we then see that, for any smooth and compactly supported function $\tau$ which is identically $1$ near the exceptional divisors, the functions $w_i := (1-\tau) v_i$ satisfy 
$$
\norm{r^a w_i}_{L^\infty} = \norm{r^a G_{R_0} \Delta_0 w_i}_{L^\infty} \leq c \norm{r^a \Delta_0 w_i}_{L^\infty} \stackrel{i \to \infty}{\To} 0.
$$
With the scaled Schauder estimate
$$
\norm{v_i}_{C^{2,\alpha}_a(B_{2 R_0}^c)} \leq \; c \left( \norm{r^a v_i}_{L^\infty(B_{R_0}^c)} + \norm{\Delta_0 v_i}_{C^{0,\alpha}_a(B_{R_0}^c)} \right) 
$$
(where we dropped the irrelevant $\epsilon$ and $b$) and the convergence over compact subsets, we therefore obtain that $\norm{v_i}_{C^{2,\alpha}_a(W)}$ goes to zero; and more generally, this remains true when $W$ is replaced by the 
complement of any compact neighbourhood of the exceptional divisors.

Next, we focus on what happens around the $j$th exceptional divisor. We consider the function $V_i := \epsilon_i^b (v_i)_{\epsilon_i}$, defined on $N(\epsilon_i)$. The bound on $\norm{v_i}_{C^{2,\alpha}_{\epsilon_i,b}(N_{j})}$ makes it possible to extract a subsequence $V_i$ which converges to a function $V_\infty$ on every compact subset of $T^* \cx P^1$. Moreover, this function $V_\infty$ is harmonic and uniformly bounded by a constant times $(1+\abs{z})^{-b}$. Since $b$ is positive, we deduce
that $V_\infty=0$, as above. Now assume that $\norm{V_i}_{C^{0}_{b}}$ remains bounded from below (up 
to subsequence). Then we can find a sequence of points $p_i$ such that $R_i := \abs{z(p_i)}$ goes to infinity and $R_i^b \abs{V_i(p_i)}$ is bounded from below. Let us rescale the metric into $\xi_i := R_i^{-2} \omega_{\epsilon_i}$ and consider the functions $W_i := R_i^b V_i$. Since $\omega_{\epsilon_i}$ is closer and closer to the Eguchi-Hanson metric and blowing down the Eguchi-Hanson metric results in the orbifold $\rl^4/\pm$, we see that $\xi_i$ converges to the flat metric on the orbifold $\rl^4/\pm$. Each $W_i$ is defined on a ball of radius $\frac{1}{8 \epsilon_i R_i}$ with respect to $\xi_i$ and we may assume that $\epsilon_i R_i$ goes to zero, in view of the convergence of $v_i$ to zero away from the exceptional divisors (proved in the first step). Since $(W_i)$ is uniformly bounded in $C^{2,\alpha}_b$ (defined with respect to $\xi_i$), we can again make it converge to a harmonic function $W_\infty$ on $\rl^4/\pm$ minus the origin. Let us denote by $\abs{x}$ the distance to the origin in $\rl^4/\pm$. By construction, $W_\infty(x_\infty)>0$ for some point $x_\infty$ with $\abs{x_\infty}=1$. But one can check that $\abs{x}^b W_\infty(x)$ is uniformly bounded, so that  the harmonic function $W_\infty$ is bound to vanish, as above. This is a contradiction, so $\norm{V_i}_{C^{0}_{b}}$ goes to zero. Scaled Schauder estimates imply that $\norm{V_i}_{C^{2,\alpha}_{b}}$ and therefore $\norm{v_i}_{C^{2,\alpha}_{\epsilon_i,b}(N_{j})}$ go to zero. Playing the same game around each component of the exceptional divisor, we obtain a contradiction.
\end{proof}

\subsection{The deformation}

We will use the following version of the implicit function theorem, whose proof is an immediate application of Banach's fixed point theorem. 

\begin{lem}\label{implicite}
Let $\appli{\Phi}{E}{F}$ be a smooth map between Banach spaces
and define $Q := \Phi - \Phi(0) - d_0\Phi$. We assume there are positive constants $q$, $r_0$ and $c$ such that
\begin{enumerate}
 \item $\norm{Q(x) -Q(y)} \leq q \norm{x-y} \left( \norm{x} + \norm{y} \right)$ for every $x$ and $y$ in $B_E(0,r_0)$ ;
 \item $d_0\Phi$ is an isomorphism with inverse bounded by $c$.
\end{enumerate}
Pick $r \leq \min(r_0,\frac1{2qc})$ and assume $\norm{\Phi(0)} \leq \frac{r}{2 c}$. Then the equation $\Phi(x)=0$
admits a unique solution $x$ in $B_E(0,r)$.  
\end{lem}

We apply this to the operator
$$
\Phi_\epsilon \; : \; \psi \mapsto \frac{\left( \omega_\epsilon + dd^c \psi \right)^2}{\omega_\epsilon \land \omega_\epsilon} - e^{f_\epsilon},
$$
between the Banach spaces $E^{a,b}_\epsilon$ and $F^{a,b}_\epsilon$ for some positive number $a$ and a positive number $b$,
with $b<2$. The linearization 
of $\Phi_\epsilon$ is $-\Delta_\epsilon$ so condition (2) stems from Lemmata \ref{surj} and \ref{estimlin}, with a constant $c$ independent of $\epsilon$. 

We need some smallness on $\Phi_\epsilon(0) = 1 - e^{f_\epsilon}$. In view of the shape of $f_\epsilon$, including (\ref{estimpotricci}), we have 
$$
\norm{f_\epsilon}_{C^{0,\alpha}_{\epsilon,a+2,b+2}} \leq c \epsilon^{3 + \frac{b}{2}},
$$
which leads to
$$
\norm{\Phi_\epsilon(0)}_{F^{a,b}_\epsilon} \leq c \epsilon^{3+\frac{b}{2}}.
$$
 
In view of condition (1), observe the non-linear term is the quadratic map given by:
$$
Q_\epsilon(\psi) = \frac{d d^c \psi \land d d^c \psi}{\omega_\epsilon \land \omega_\epsilon}
$$
For every $\psi_1$ and $\psi_2$ in $C^{2,\alpha}_{\epsilon,a,b}$, we have the following estimate:
$$
\norm{Q_\epsilon(\psi_1) - Q_\epsilon(\psi_2)}_{C^{0,\alpha}_{\epsilon,a+2,b+2}} \leq c \, \epsilon^{-b-2}
\norm{\psi_1 - \psi_2}_{C^{2,\alpha}_{\epsilon,a,b}}
\left( \norm{\psi_1}_{C^{2,\alpha}_{\epsilon,a,b}} + \norm{\psi_2}_{C^{2,\alpha}_{\epsilon,a,b}}\right),
$$
where $c$ does not depend on $\epsilon$. This uniform $C^{0,\alpha}_{\epsilon,a+2,b+2}$ control is pretty clear on $W$,
without any $\epsilon$ in the constant, indeed. Near the divisors, we need to compensate the (small) weight, hence this unpleasant 
$\epsilon^{-b-2}$. It is straightforward to extend to the case where $\psi_1$ and $\psi_2$ in $E^{a,b}_\epsilon$: 
$$
\norm{Q_\epsilon(\psi_1) -Q_\epsilon(\psi_2)}_{F^{a,b}_\epsilon} \leq q_\epsilon \norm{\psi_1-\psi_2}_{E^{a,b}_\epsilon} \left( \norm{\psi_1}_{E^{a,b}_\epsilon} + \norm{\psi_2}_{E^{a,b}_\epsilon} \right),
$$
with $q_\epsilon= c \, \epsilon^{-b-2}$. In order to use Lemma \ref{implicite}, we compare $\norm{\Phi_\epsilon(0)}_{F^{a,b}_\epsilon}$
to $\frac1{q_\epsilon}$:
\begin{equation}\label{comparison}
\norm{\Phi_\epsilon(0)}_{F^{a,b}_\epsilon} \leq c \epsilon^{1-\frac{b}{2}} q_\epsilon^{-1}.
\end{equation}
Then, since $b<2$, we see that for small $\epsilon$, $\norm{\Phi_\epsilon(0)}_{F^{a,b}_\epsilon}$ is much smaller than $\frac1{q_\epsilon}$ and we can use Lemma \ref{implicite} to prove the following theorem. Recall that we denote the exceptional divisors by $E_1,\dots,E_p$ and the Poincar\'e dual of $E_j$ by $PD[E_j]$. Observe we can pick different deformation parameters $\epsilon_j = a_j \epsilon$ around the divisors $E_j$, in order to get a larger range of examples. 

\begin{thm}\label{thm:alh}
Let $a_1, \dots, a_8$ denote some positive parameters. Then for every small enough positive number $\epsilon$, there 
is a Ricci-flat K\"ahler form $\omega$ on $\hat{X}$ in the cohomology 
class $[\omega_0]  - \epsilon \sum a_j PD[E_j]$. These provide ALH gravitational 
instantons asymptotic to $\rl_+ \times \tor^3$: $\omega = \omega_0 + \OO(r^{-\infty})$.  
\end{thm}

The notation $\OO(r^{-\infty})$ denotes a function decaying faster than any (negative) power of $r$ (cf. appendix, after Lemma \ref{sauts}). By working with exponential weights, we may prove that the decay rate to the flat metric is indeed exponential. 

\begin{proof}
We can apply Lemma \ref{implicite} to solve the Monge-Amp\`ere equation (\ref{mongeampere}) for small enough $\epsilon$, which provides a $(1,1)$-form $\omega = \omega_\epsilon + dd^c \psi$, with $\psi \in E^{a,b}_\epsilon$. Since $\omega$ is asymptotic to
$\omega_\epsilon$, it is positive outside a compact set. From (\ref{mongeampere}), we know that $\omega$ is everywhere non-degenerate,
so it must remain positive on $\hat{X}$ : it is a K\"ahler form. Moreover, its Ricci form is given by $\rho_\epsilon - \frac12 dd^c f_\epsilon$, which vanishes in view of our choice of $f_\epsilon$, so $\omega$ is Ricci-flat. Finally, a Ricci-flat K\"ahler structure on a simply-connected four-dimensional manifold is the same as a hyperk\"ahler structure, so we only need to check that $\hat{X}$ is simply connected. First, observe that $X=\rl \times \tor^3 /\pm$ retracts onto $\tor^3 /\pm$, which is covered by two open sets $U_1$ and $U_2$ with connected intersection, such that $U_1$ and $U_2$ are homeomorphic to $[0,\frac12) \times \tor^2 / \sim $ where, for any $x \in \tor^2$, $(0,x) \sim (0,-x)$. Since $[0,\frac12) \times \tor^2 / \sim $ retracts onto $\tor^2/\pm$, which is homeomorphic to the $2$-sphere, we eventually see that $X$, and therefore $\hat{X}$, is simply-connected.
\end{proof}

\section{Other similar constructions}

\subsection{New ALH Ricci-flat manifolds}
\label{sec:new-alh-ricci}
In the previous example, it is natural to try and replace $\tor^3 =
\rl^3 /\ir^3$ by another compact flat orientable three-manifold.  Let
$F_2$ and $F_{2,2}$ denote the smooth flat three-manifolds obtained as
$F_2 := \tor^3/\sigma$ and $F_{2,2}:=\tor^3/\langle\sigma,\tau\rangle$, where $\sigma$ and $\tau$ are
the two commuting involutions
\begin{align*}
  \sigma(x,y,z) &=(x+\frac12,-y,-z+\frac12),\\
  \tau(x,y,z) &=(-x,-y+\frac12,z+\frac12).
\end{align*}
Then $\rl \times F_2$ is naturally a
complex flat K\"ahler manifold---a quotient of $\cx^2$, indeed. More
specifically, if $t$ is the coordinate along $\rl$, we take $t+ix$ and
$y+iz$ as complex coordinates.

We then consider the complex flat K\"ahler orbifold $X_2 := \rl \times F_2 /
\pm$. The reader may check that this involution is well defined and has
four fixed points, yielding rational double points. We blow them up to
obtain the complex manifold $\hat{X_2}$.

\begin{prop}\label{prop:alh2}
  Let $a_1, \dots, a_4$ denote some positive parameters. Then for
  every small enough positive number $\epsilon$, there is a Ricci-flat K\"ahler
  form $\omega$ on $\hat{X_2}$ in the cohomology class $[\omega_0] - \epsilon \sum a_j
  PD[E_j]$. These provide ALH manifolds asymptotic to the flat metric
  on $\rl_+ \times F_2$: $\omega = \omega_0 + \OO(r^{-\infty})$.
\end{prop}

In this statement, $\omega_0$ denotes again the pull-back of the flat
K\"ahler form on $X_2$ and $r=\abs{t}$. Actually the proposition is a
direct consequence of theorem \ref{thm:alh}, since one can perform
first the quotient by $\pm$ and then by $\sigma$. Indeed the involution $\sigma$
acts freely on $\setR\times\tor^3/\pm$ and on its desingularization, say $\hat X_1$; if
the K\"ahler class in theorem \ref{thm:alh} is invariant, then it is
obvious that the whole construction can be made $\sigma$ invariant, so the
resulting metric descends on $\hat{X_2}$.

It follows that the fundamental group of $\hat{X}_2$ is $\setZ_2$. The
metrics are not hyperk\"ahler because the holomorphic symplectic form
$\Omega$ on $\hat X_1$ satisfies $\sigma^*\Omega=-\Omega$, so there is only a multivalued
symplectic form on $\hat{X_2}$. (This is also apparent on the flat
model $\setR\times F_2$, whose holonomy is not in $SU(2)$.)

\smallskip
Finally the involution $\tau$ is real with respect to the above choice of
complex structure, and acts freely on $X_2$ and $\hat X_2$, on which
it exchanges the four curves $E_j$, say for example $\tau E_1=-E_2$ and
$\tau E_3=-E_4$. This leads to:
\begin{prop}\label{prop:alh22}
  With the same notations as above, if $a_1=a_2$ and $a_3=a_4$ then
  the metric of proposition \ref{prop:alh2} is $\tau$ invariant so
  descends to a Ricci flat, locally K\"ahler metric on $\hat
  X_{2,2}:=\hat X_2/\tau$. This is an ALH metric with an end asymptotic
  to $\setR_+\times F_{2,2}$.
\end{prop}
Again the whole construction can be made $\sigma$ and $\tau$ invariant (in
particular, looking for a $\tau$ invariant potential), so the proposition
is immediate.

\subsection{Hitchin's ALF gravitational instanton}\label{hitchin}

In \cite{Hit}, N. Hitchin built a hyperk\"ahler structure on the desingularization
$\hat{X}$ of $\rl^3 \times \sph^1/\pm$ through twistor theory. Beware $\sph^1$ is again seen as 
$\rl/\ir$ (so the involution is not an antipodal map). Our direct analytical approach 
gives another construction of this hyperk\"ahler manifold.

\begin{thm}\label{thm:d2}
Let $a_1, a_2$ denote some positive parameters. Then for every small enough positive number $\epsilon$, there 
is a Ricci-flat K\"ahler form $\omega$ on $\hat{X}$ in the cohomology 
class $[\omega_0]  - \epsilon \sum a_j PD[E_j]$. These provide ALF gravitational 
instantons asymptotic to $\rl^3 \times \sph^1 /\pm $: $\omega = \omega_0 + \OO(r^{-3+\delta})$, for every positive $\delta$.
\end{thm}

\begin{proof}
The proof follows the same lines, so we just point out the necessary adaptations. We work with $0<a<1$, so that weighted analysis ensures the Laplacian is an isomorphism. In the definition of $E^{a,b}_\epsilon$, the $\rl \tilde{r}$ must therefore be dropped. 
The analysis can then be done similarly and we just need to check that $\hat{X}$ is simply connected: this is immediate, for $\rl^3 \times \sph^1/\pm$ turns out to be contractible.
\end{proof}

\subsection{ALG gravitational instantons}

In order to build ALG examples with the same technique, we may start from $\rl^2 \times \tor^2$
and consider `crystallographic' quotients. The basic example is $X_2 = \rl^2 \times \tor^2 /\pm$, which is
a complex flat K\"ahler orbifold with four rational double points. When $\tor^2$ is obtained from
a square lattice in $\rl^2$, we may also consider $X_2 = \rl^2 \times \tor^2 /\ir_4$, where the action of
$\ir_4$ is induced by the rotation of angle $\frac{\pi}{2}$ on both factors. In this case, there are two 
$\cx^2 /\ir_4$ singularities and one $\cx^2 /\ir_2$ singularity. Similarly, starting from a 
hexagonal lattice and using rotations of angle $\frac{\pi}{3}$ and $\frac{\pi}{6}$, we may work with 
$X_k = \rl^2 \times \tor^2 /\ir_k$ for $k=3$ or $6$. The orbifold $X_3$ has three $\cx^2 /\ir_3$ singularities,
while $X_6$ has one $\cx^2 /\ir_6$ singularity, one $\cx^2 /\ir_2$ singularity and one $\cx^2 /\ir_3$ singularity. 
In any case, we may blow up the singularities to get the smooth complex manifold $\hat{X_k}$, $k=2,3,4,6$.
Every $\cx^2/\ir_k$ singularity can be endowed with an asymptotically locally Euclidean (ALE) Ricci-flat 
K\"ahler metric: the Gibbons-Hawking or multi-Eguchi-Hanson metrics \cite{GH,Joy}. We may use them in the 
gluing procedure. We do not have an explicit K\"ahler 
potential $\phi$ but for instance Theorem 8.2.3 in \cite{Joy} gives a potential $\phi = \frac{\abs{z}^2}{2} + \OO(\abs{z}^{-2})$, which is what we need. 

\begin{thm}
Pick $k=2,3,4,6$ and let $a_1, \dots, a_p$ denote some positive parameters ($p$ is the number of singularities). 
Then for every small enough positive number $\epsilon$, there 
is a Ricci-flat K\"ahler form $\omega$ on $\hat{X}$ in the cohomology 
class $[\omega_0]  - \epsilon \sum a_j PD[E_j]$. These provide ALG gravitational 
instantons asymptotic to $\rl^2 \times \tor^2 /\ir_k $: $\omega = \omega_0 + \OO(r^{-k-2+\delta})$, for every positive $\delta$.
\end{thm}

\begin{proof}
Again, we just point out the necessary adaptations. To begin with, we may check that $\hat{X_k}$ is simply connected: this follows from the fact that $\tor^2 /\ir_k $ is homeomorphic to the two-sphere (it is for instance a consequence of the Gauss-Bonnet formula for closed surfaces with conical singularities). In view of weighted analysis, we can work with $0<a<k$ (because there is no harmonic function on $\rl^2/\ir_k$ that decays like $r^{-a}$) and in the definition of $E^{a,b}_\epsilon$, the $\rl \tilde{r}$ summand has to be replaced by $\rl \widetilde{\log r}$ (a smooth function equal to $\log r$ outside a compact set). 
\end{proof}

\subsection{ALF gravitational instantons of dihedral type}

In this section, we start from the Taub-NUT metric $g_{TN}$ on $\rl^4$. It is given by the following explicit formulas. We refer to \cite{Leb} for details. To begin with, we identify $\rl^4$ with $\cx^2$, with complex coordinates $w_1, w_2$. The Hopf fibration $\appli{\pi=(x_1,x_2,x_3)}{\cx^2}{\rl^3}$ is given by 
$$
x_1 = 2 \re (w_1 \bar{w_2}), \quad x_2 = 2 \im (w_1 \bar{w_2}), \quad x_3 = \abs{w_1}^2 - \abs{w_2}^2.
$$
Let us fix a positive number $m$ and define 
\begin{equation}
V = 1 + \frac{2m}{\abs{x}}, \quad \theta = 8m\frac{\im (\bar{w_1} dw_2)}{\abs{w}^2}.\label{eq:8}
\end{equation}
Then $\frac{\theta}{4m}$ is a connection one-form on the Hopf fibration, with curvature $\frac{*_{\rl^3} dV}{4m}$,
and the Taub-NUT metric is given by 
\begin{equation}
g_{TN} = V (dx_1^2 + dx_2^2 + dx_3^2) + \frac{1}{V} \theta^2.\label{eq:7}
\end{equation}
It turns out to be a complete K\"ahler metric, with respect to the complex structure $I$ mapping $dx_1$ to $dx_2$
and $dx_3$ to $\frac{\theta}{V}$. The corresponding K\"ahler form is 
$$
\omega_{TN} = V dx_1 \land dx_2 + dx_3 \land \theta.
$$
Moreover, it is endowed with a parallel symplectic $(2,0)$ form:
$$
\Omega = (V dx_2 \land dx_3 + dx_1 \land \theta) + i (V dx_3 \land dx_1 + dx_2 \land \theta).
$$ 
This holomorphic symplectic structure is in fact isomorphic to the standard one on $\cx^2$ \cite{Leb}.
The Taub-NUT metric is therefore hyperk\"ahler. 

This hyperk\"ahler structure is preserved by an action of the binary
dihedral group $D_k$ (of order $4(k-2)$) for every $k > 2$.
Explicitly, we see $D_k$ as the group generated by the following diffeomorphisms $\tau$ and $\zeta_k$ of 
$\cx^2$:
$$
\tau (w_1,w_2) = (\bar{w_2},-\bar{w_1}), \quad \zeta_k(w_1,w_2) = (e^{\frac{i\pi}{k-2}} w_1, e^{\frac{i\pi}{k-2}} w_2). 
$$
And the reader may check that this action preserves the whole hyperk\"ahler structure. 

We then let $X$ be the orbifold obtained as the quotient of the Taub-Nut manifold by this action of $D_k$. It has one 
complex singularity, isomorphic to the standard $\cx^2/D_k$ (with $D_k$ in $SU(2)$). Let us denote the minimal resolution 
of $X$ by $\hat{X}$. Again, we need approximately Ricci-flat metrics on $\hat{X}$. Near the exceptional divisor, it is natural 
to glue one of the $D_k$ ALE gravitational instantons introduced by P. B. Kronheimer in \cite{Kro}. This yields a potential 
$\phi_{ALE} = \frac{\abs{z}^2}{2} + \OO(\abs{z}^{-2})$. If we implemented the same gluing as above, we would be in trouble, basically because $g_{TN}$ is not flat. Technically, we would end up with $f_\epsilon = \OO(\epsilon)$ and therefore the exponent in (\ref{comparison}) would be $-\frac{b}{2}$ instead of $1-\frac{b}{2}$; since $b$ has to be positive, this exponent would be negative
and we would never find a ball where to perform the fixed point argument. So we need to refine the gluing. 

Near the origin, we can find complex coordinates $s=(s_1,s_2)$ in which $\omega_{TN}$ is the 
standard flat K\"ahler form $\omega_0 = dd^c\frac{\abs{s}^2}{2}$ up to $\OO(\abs{s}^2)$ and, more precisely, we can find a potential $\phi_{TN}$ for $\omega_{TN}$ with the expansion 
\begin{equation}\label{exptn}
\phi_{TN} = \frac{\abs{s}^2}{2} + \theta_4(s) + \OO(\abs{s}^{5}),
\end{equation}
where $\theta_4(s)$ is a $D_k$-invariant quartic expression in $s$ (and $\bar{s}$).
Moreover, since $\omega_{TN}$ is Ricci-flat, we have the Monge-Amp\`ere equation 
$$
dd^c \log \left( \frac{\omega_{TN}^2}{\omega_0^2}  \right) = 0,
$$
which can be expanded into $d d^c \Delta_{\omega_0} \theta_4 = \OO(\abs{s})$. Since $\Delta_{\omega_0} \theta_4$ is a quadratic form, it is bound to vanish, so $\theta_4$ is harmonic. We then identify a neighbourhood of $0$ in $\hat{X}$ with a large domain in the $D_k$ ALE gravitational instanton, in the same manner as previously, with an $\epsilon$-dilation $s \mapsto \epsilon z$. Then $\theta_4(s)= \epsilon^4 \theta_4(z)$. Since $\theta_4$ is harmonic with respect to the flat metric, we see that $\Delta_{ALE} \theta_4 = \OO(\abs{z}^{-2})$. From weighted analysis, we may then find a $\Delta_{ALE}$-harmonic function $h_4$ with $h_4 = \theta_4(z) + \OO(\abs{z}^{-2})$. Instead of gluing the Taub-NUT metric with the (scaled) ALE metric, we will patch the Taub-NUT potential $\phi_{TN}$ together with $\epsilon^2(\phi_{ALE} + \epsilon^2 h_4)$, namely the approximately Ricci-flat metric $\omega_\epsilon$ we use in this context is given by the potential 
$$
\phi_{\epsilon} := \chi_\epsilon \left[ \epsilon^2(\phi_{ALE} + \epsilon^2 h_4) \right] 
+ (1-\chi_\epsilon) \phi_{TN},
$$
with a cutoff function $\chi_\epsilon$ like in (\ref{patch}). The $(1,1)$-form $\omega_{loc,\epsilon} = dd^c (\phi_{ALE} + \epsilon^2 h_4)$ therefore plays the role of the Eguchi-Hanson metric $\omega_{EH}$ in this context. Beware it depends on $\epsilon$ and defines a K\"ahler metric (a priori) only on some ball $\set{\abs{z} \leq \frac{c}{\epsilon}}$, owing to the estimate 
\begin{equation}\label{perturb}
\abs{\omega_{loc,\epsilon} - \omega_{ALE}} \leq c \epsilon^2 \abs{z}^2.
\end{equation}
We need a control on the function $f_\epsilon$ given by (\ref{potricci}). Since Taub-NUT is Ricci flat, $f_\epsilon$ vanishes 
for $\rho \geq 2\sqrt{\epsilon}$. On $\set{\rho \leq \sqrt{\epsilon}}$, we may use the Ricci-flat $\omega_{ALE}$ and observe 
$$
e^{- f_\epsilon} = \frac{\omega_{loc,\epsilon}^2}{\omega_{ALE}^2} 
= 1 + 2 \epsilon^2 \Delta_{ALE} h_4 + \epsilon^4 \frac{(dd^c h_4)^2}{\omega_{ALE}^2}
= 1 + 0 + \epsilon^4 \OO(\abs{z}^4) = \OO(\rho^4),
$$
which results in $\abs{\nabla^k f_\epsilon} \leq c(k) \epsilon^{2-\frac{k}{2}}$. Finally, on the transition area
$\set{\sqrt{\epsilon} \leq \rho \leq 2\sqrt{\epsilon}}$, we use the expansions (\ref{exptn}) and 
$$
\phi_{ALE} = \frac{\abs{z}^2}{2} + \OO(\abs{z}^{-2}) 
= \epsilon^{-2} \frac{\abs{s}^2}{2} + \epsilon^2 \OO(\abs{s}^{-2})
$$
to obtain 
$$
\omega_\epsilon - \omega_{TN} = dd^c \left( \epsilon^4  \chi_\epsilon \OO(\abs{s}^{-2}) 
+ \chi_\epsilon \OO(\abs{s}^{5}) \right).
$$
Note that without the trick consisting in plugging this function $h_4$ into the potential, 
the last exponent would have been $4$ instead of $5$, resulting in the bad estimate 
$f_\epsilon = \OO(\epsilon)$. Instead, here, we find 
$$
\abs{\nabla^k(\omega_\epsilon - \omega_{TN})} \leq c(k) \epsilon^{\frac{3-k}{2}}
$$
and eventually
$$
\abs{\nabla^k f_\epsilon} \leq c(k) \epsilon^{\frac{3-k}{2}}.
$$
If we follow the proof detailed above, this leads to an exponent $\frac{1-b}{2}$ in (\ref{comparison}), which is good enough since we can choose any $b$ in $(0,1)$. The other arguments can be adapted. In particular, the proof of Lemma \ref{estimlin} still works, because $\omega_{loc,\epsilon}$ gets closer and closer to the ALE metric on larger and larger domains, cf. (\ref{perturb}). 

\begin{thm}\label{thm:dk-alf}
For every small $\epsilon$, there is a Ricci-flat K\"ahler form $\omega$ on $\hat{X}$ in the cohomology class $[\omega_{TN}]  - \epsilon^2 PD[E]$, where $PD[E]$ denotes the Poincar\'e dual of the exceptional divisor. These provide ALF gravitational 
instantons : $\omega = \omega_{TN} + \OO(r^{-3+\delta})$, for every positive $\delta$.
\end{thm}

These ALF gravitational instantons are of dihedral type in the sense
of \cite{Mi2}. For $k=3$ (resp. $k=4$), they have the same asymptotics
as the Atiyah-Hitchin metric, that is the $D_0$ ALF gravitational
instanton (resp. its double cover, the $D_1$ ALF gravitational
instanton), with the difference that they have positive mass: their
metric is asymptotic to $g_{TN}$ with a positive parameter $m$, in
contrast with the Atiyah-Hitchin metric where the model at infinity is
Taub-NUT with a negative parameter $m$. As we shall see in the next
section, these are the only two cases where this can happen.  Also
note that the examples we build presumably coincide with the $D_k$ ALF
metrics of Cherkis-Dancer-Hitchin-Kapustin \cite{CH,CK1,Dan}.

\begin{rem*}
  The class of ALF gravitational instantons of cyclic type (whose
  boundary is fibered over $\sph^2$) is completely classified \cite{Mi2}:
  it is the class of multi-Taub-NUT metrics, with boundaries at
  infinity $\sph^3$ quotiented by the cyclic group $A_k$ ($k\geq 0$, the
  $k=0$ case is the Taub-Nut metric on $\setR^4$ described above). One
  should add one special case, the flat space $\setR^3\times \sph^1$ which can be
  numbered $A_{-1}$ (this fits well with several formulas in the next
  section).

  As mentioned to us by S. Cherkis, one can also construct $D_k$ ALF
  gravitational instantons starting from an $A_{2k-5}$ ALF
  gravitational instanton (a multi-Taub-NUT metric associated to a
  symmetric configuration of $2k-4$ points), and taking the quotient
  by an involution with two fixed points. The same technique applies
  and provides a hyperk\"ahler metric on the desingularization. The
  special case $k=2$ leads to the construction of a $D_2$ ALF
  gravitational instanton (conventionally the Hitchin metric) from a
  $A_{-1}$ one, that is from $\setR^3\times \sph^1$: this is the construction in
  section \ref{hitchin}.
\end{rem*}

\section{A formula for the Euler characteristic}
\label{sec:form-euler-char}

Let $X$ be an ALF gravitational instanton of dihedral type or cyclic
type. Near infinity, one has $X\simeq (A,+\infty)\times S$, where $S$ has a circle
fibration over $\Sigma=\setR P^2$ (dihedral case) or $\Sigma=\sph^2$ (cyclic
case). Moreover the metric $g$ has the following asymptotics:
$$ g \simeq dr^2 + r^2 \gamma + \theta^2 , $$
where $\theta$ is a connection 1-form on the circle bundle (or its double
covering in the dihedral case), and $\gamma$ is the horizontal metric lifted
from the standard metric on $\Sigma$. We have the following behavior for
the second fundamental form $\mathbb{I}$ and the curvature $R$:
\begin{equation}
 |\mathbb{I}| = \OO(\tfrac 1r), \quad |R|=\OO(\tfrac 1{r^3}) . \label{eq:4}
\end{equation}
There are well known formulas giving the Euler characteristic and
signature of $X$ in terms of the integrals of characteristic classes
on a large domain $D_\rho=\{r\leq \rho\} \subset X$ and boundary terms: for a gravitational
instanton, there remains only
\begin{align*}
  \chi &= \frac 1{8\pi^2} \int_{D_\rho} |W_-|^2
      + \frac 1{12\pi^2}\int_{\partial D_\rho} \mathfrak{T}(\mathbb{I}\land(\mathbb{I}\land\mathbb{I}+3R)) ,\\
  \tau &= \frac 1{12\pi^2} \int_{D_\rho} -|W_-|^2
      + \frac 1{12\pi^2} \int_{\partial D_\rho} \mathfrak{S}(\mathbb{I}(\cdot,R(\cdot,\cdot)n))
      + \eta(\partial D_\rho).
\end{align*}
Here $n$ is the normal vector, $\mathfrak{T}$ and $\mathfrak{S}$ are linear operations
which we do not need to write down explicitly, since from the control
(\ref{eq:4}) and the fact that the volume of $\partial D_\rho$ is $\OO(\rho^2)$, we
obtain that all boundary integrals go to zero when $\rho$ goes to
infinity. Finally this implies the following form of the
Hitchin-Thorpe inequality:
\begin{equation}
  \label{eq:5}
  2\chi+3\tau = \lim_{\rho\to\infty} \eta(\partial D_\rho) .
\end{equation}
For the gravitational instanton $X$, if $X\neq \setR^3\times \sph^1$ we have
$b_1(X)=b_3(X)=0$, and the intersection form is negative definite (see
\cite{HHM}, this follows immediately from the fact that the relevant
cohomology classes can be represented by $L^2$ harmonic forms), so it
follows that $\tau=-(\chi-1)$.  On the other hand, since the $\eta$-invariant
is conformally invariant, the limit in (\ref{eq:5}) is the adiabatic
limit:
$$ \eta_{\mathrm{ad}}(S) := \lim_{r\to\infty} \eta(\gamma+\frac 1{r^2}\theta^2) . $$
Therefore we obtain the following result:
\begin{thm}\label{thm:euler-eta}
  For an ALF gravitational instanton $X\neq \setR^3\times \sph^1$, with boundary $S$, one has
  \begin{equation}
 \chi(X) = 3\big( 1-\eta_{\mathrm{ad}}(S) \big) .\label{eq:6}
 \end{equation}
\end{thm}
The calculation of the adiabatic limit of the $\eta$-invariant is well
known, but we can also deduce it from the theorem: in both cyclic and
dihedral cases, we have examples obtained by desingularizing the
quotient of $\setC^2$ with the Taub-NUT metric by the cyclic group $A_k$
(this gives the multi-Taub-NUT metrics), or the dihedral group $D_k$
(the metrics coming from theorem \ref{thm:dk-alf}). This results in a
$k$-dimensional 2-cohomology and therefore $\chi=k+1$ and
\begin{equation}
\eta_{\mathrm{ad}}=\frac {2-k}3.\label{eq:9}
\end{equation}
In the dihedral case, the formula extends immediately to the $D_2$
case, which is Hitchin's metric on the desingularization of
$\setR^3\times \sph^1/\pm$. In this way the values of (\ref{eq:9}) for $k\geq 2$ give
the adiabatic $\eta$ invariant for all possible boundaries $S$ of an ALF
gravitational instanton. Nevertheless observe that the sign of the
$\eta$-invariant is changed if the orientation of $S$ is changed.

From the theorem \ref{thm:euler-eta}, since one must have $\chi\geq 1$, one
deduces the constraint
\begin{equation}
 \eta_{\mathrm{ad}}(S) \leq \frac 23 .\label{eq:10}
\end{equation}
From the values obtained in (\ref{eq:9}), we see that the only three
cases where the boundary $S$ of a dihedral ALF gravitational
instanton, endowed with the opposite orientation, can be filled by
another gravitational instanton, are $k=2$, $3$ or $4$. Indeed, for
$k=4$, the $D_0$ gravitational instanton (the Atiyah-Hitchin metric)
has the same boundary as the $D_4$ instanton, but with the opposite
orientation; observe that since it retracts on a $\setR P^2$ it has $\chi=1$
and $\eta_{\mathrm{ad}}=\frac23$, so the formulas (\ref{eq:6}) and
(\ref{eq:9}) remain true. For $k=3$, we have the same phenomenon with
the $D_1$ ALF gravitational instanton (the double cover of the $D_0$
one) which has the same boundary as the $D_3$ one up to
orientation. Finally for $k=2$, one has $\eta_{\mathrm{ad}}=0$ and the
opposite orientation is obtained by the same space, since the flat
space $\setR^3\times \sph^1/\pm$ admits an orientation reversing isometry.

We have proved:
\begin{cor}
  There is no dihedral ALF gravitational instantons with boundary
  equal to $\sph^3/D_k$ with negative orientation for $k>4$.
\end{cor}

Let us observe from the ansatz (\ref{eq:8}) (\ref{eq:7}) for the
Taub-NUT metric that the orientation of the boundary $S$ depends on
the sign of the mass $m$. Specifying the sign of the mass is therefore
the same as specifying the orientation of the boundary $S$. In the
cyclic case, all ALF gravitational instantons but $\setR^3\times \sph^1$ have
positive mass \cite{Mi1,Mi2}. In the dihedral case, the corollary
implies that all ALF gravitational instantons have positive mass, with
the only exceptions of $D_0$ or $D_1$ asymptotics (negative mass), or
$D_2$ asymptotics (zero mass).

Finally, remind that, if in the cyclic case the ALF gravitational
instantons are completely classified \cite{Mi2}, the classification is
still an open problem in the dihedral case: at least the corollary
tells us that there is no possible new class with negative mass in the
$D_k$ case for $k>4$.

\section{Other ALH Ricci-flat K\"ahler examples}
\label{sec:other-alh-ricci}

There are six oriented compact flat 3-manifolds \cite{Wol}: the torus
$\bT^3$, four quotients $F_j=\bT^3/\setZ_j$ for $j=2,3,4,6$ and the quotient
$F_{2,2}=\bT^3/\setZ_2\times\setZ_2$.  In \S~\ref{sec:new-alh-ricci}, we constructed by
quotient a K\"ahler Ricci-flat metric with one ALH end asymptotic to $\setR\times
F_2$. In this section we will exhibit similar examples with one end
asymptotic to $\setR\times F_j$ for $j=3,4,6$. This amounts to construct
suitable rational elliptic surfaces with finite group action.

Choose $\zeta_j=\exp(2\pi i/j)$ and a flat 2-torus $\bT^2$ with an action of
$\setZ_j$. Then the flat manifold $F_j$ is obtained as the quotient of
$\bT^3=\sph^1\times\bT^2$ by the diagonal action of $\setZ_j$ obtained by
multiplication by $\zeta_j$ on both factors. The flat metric
\begin{equation}
dt^2+dx^2+dy^2+dz^2\label{eq:2}
\end{equation}
on $\setR\times\bT^3$ descends to a flat K\"ahler metric on $\setR\times F_j$, but the
holomorphic-symplectic form $\Omega=(dt+idx)\land (dy+idz)$, which has a simple
pole at infinity in the compactification $\bP^1\times\bT^2$, becomes
$j$-multivalued in the quotient: the metric is not hyperk\"ahler since
the monodromy at infinity is not a subgroup of $SU(2)$.  (For the last
flat 3-manifold the monodromy is not a subgroup of $U(2)$ so one can
not hope to construct K\"ahler examples, but one can still hope to
construct actions leading to ALH Ricci flat examples.)

We start from a rational elliptic surface $X_j$ with:
\begin{itemize}
\item if $j=3$, three singular fibres of type $IV$;
\item if $j=4$, four singular fibres of type $III$;
\item if $j=6$ six singular fibres of type $II$.
\end{itemize}
A glance at the table in \cite[p. 206]{Mir90} shows that such
surfaces exist. One can construct them in a concrete way using the
Weierstrass model: if $L=\OO_{\bP^1}(1)$, and $g_2$ and $g_3$ are
holomorphic sections of $L^4$ and $L^6$, then the surface
\begin{equation}
  \label{eq:1}
  y^2z = 4x^3 - g_2 xz^2 - g_3 z^3 \quad \text{ in } \bP(L^2\oplus L^3\oplus\OO_{\bP^1})
\end{equation}
is a rational elliptic surface. In case $g_3=0$ and $g_2$ has four
simple zeros, one gets $X_4$; if $g_2=0$ and $g_3$ has six simple
zeros, one gets $X_6$; if $g_2=0$ and $g_3$ has three double zeros
one gets $X_3$. Moreover we can choose $g_2$ and $g_3$ so that
$X_j$ has an action of $\setZ_j$ over $\bP^1$ which permutes the singular
fibres. For example we take the standard action of $\setZ_j$ on $\bP^1$ by
$z\mapsto\zeta_jz$  and we use $g_2(u)=u^4-1$ for
$j=4$, $g_3(u)=u^6-1$ for $j=6$ and $g_3(u)=(u^3-1)^2$ for $j=3$.

Given any fibre, there is a holomorphic symplectic form on $X_j$ with
a simple pole along this fibre, giving a section of $K(F)$. We choose
$\Omega\in H^0(X_j,K(F))$ the symplectic form with a simple pole over the
fibre at infinity, so that near infinity one has $\Omega\sim \frac{dz}z\land dv$,
where $v$ is a coordinate on the fibre at infinity.

The action of $\setZ_j$ on $\bP^1$ has fixed points $0$ and $\infty$. The
action can be chosen so that it is free on the fibre over the origin
(translation), but has fixed points on the fibre at infinity, giving
Kleinian singularities of type $\setC^2/\setZ_j$ on the quotient $X_j/\setZ_j$.
The minimal desingularization $\hat X_j$ is again an elliptic surface
over $\bP^1$, with a multiple fibre of order $j$ over the origin, a
singular fibre of type $IV^*$ ($j=3$), $III^*$ ($j=4$) or $II^*$
($j=6$) over the point at infinity, and similarly a singular fibre of
type $IV$, $III$ or $II$ over $u=1$. Moreover, the section $\Omega^j$
descends as a section $\hat \Omega\in H^0(\hat X_j,K^j(F))$ which does not
vanish on $\hat X_j$ and has a simple pole over $\infty$ (in other words,
$\hat \Omega^{\frac 1j}$ is a multivalued holomorphic symplectic form
outside the fibre over $\infty$).

Given a K\"ahler form $\omega$ which is asymptotic to (\ref{eq:2}), an ALH
K\"ahler Ricci flat metric on $\hat{X_j}$ is given by a solution of the
Monge-Amp\`ere equation
\begin{equation}
  \label{eq:3}
  (\omega+i\partial\bar \partial f)^2 = \Omega^{\frac 1j}\land \overline{\Omega^{\frac 1j}} ,
\end{equation}
where $f$ has exponential decay on the end $\setR\times F_j$.
One can either solve directly on $\hat{X_j}$ or find a $\setZ_j$-invariant
solution on $X_j$: this amounts to solving the Monge-Amp\`ere equation
for cylindrical ends, and we refer to \cite{TY1,Joy,Kov}. More
specifically the case of $X_j$ is done in \cite{Hei}.

Using the same construction, one can recover the ALH example $\hat
X_2$ of \S~\ref{sec:new-alh-ricci} for $j=2$, starting from a rational
elliptic surface with two singular fibres of type $I_0^*$ with an
action of $\setZ_2$. In that case, our desingularization procedure of the
flat metric $\setR\times\bT^3/\setZ_2\times\setZ_2$ gives a good approximation of certain
solutions of (\ref{eq:3}). This flat model is no more available for
$j=3,4,6$.

\begin{rem*}
  It might seem disappointing that these non-hyperk\"ahler examples
  occur as finite quotients of hyperk\"ahler manifolds. It turns out to
  be a general fact: any Ricci-flat K\"ahler four-manifold with ALE,
  ALF, ALG or ALH asymptotics is bound to have a hyperk\"ahler finite
  cover. To see why, observe that such a manifold $M$ has a flat
  canonical bundle (since $\ric=0$), determined by a representation
  $\rho$ of $\pi_1(M)$ in $\cx$. Building a hyperk\"ahler finite cover
  amounts to finding a subgroup $G$ of $\pi_1(M)$ of finite index and on
  which $\rho$ is trivial. Now, since $\ric=0$, the Weitzenb\"ock formula
  ensures the $L^2$ cohomology vanishes in degree $1$. In terms of
  standard De Rham cohomology, this implies \cite{And} that the image
  of the natural map $H^1_c (M) \to H^1(M)$ is trivial. Since the
  complement of a compact set in $M$ is diffeomorphic to $\rl_+ \times S$
  for some compact $3$-manifold $S$, this means $H^1(M)$ injects into
  the cohomology space $H^1(S)$ of the `boundary at infinity' $S$ or
  in other words $H_1(S,\rl)$ surjects onto $H_1(M,\rl)$. In all
  ALE,F,G,H asymptotics, $H_1(S,\rl)$ is generated by a finite number
  of loops $\gamma_i$ for which some iterate $\gamma_i^{k_i}$ acts trivially on
  the canonical bundle of $M$.  So the subgroup $G$ of $\pi_1(M)$
  generated by the derived subgroup $[\pi_1(M),\pi_1(M)]$ and the
  $\gamma_i^{k_i}$'s has the required properties.
\end{rem*}

\appendix

\section{Analysis in weighted spaces}
\label{sec:analys-weight-spac}

Our construction relies on a few facts about the behaviour of the Laplacian on functions in complete non-compact Riemannian manifolds
$(M,g)$ with prescribed asymptotics. Let us sum up the theory. 

Basically, we assume here the existence of a compact domain $K$ in $M$ such that $M \backslash K$ has finitely many connected components which, up to a finite covering, are diffeomorphic to the complement of the unit ball in  $\rl^{m} \times \tor^{4-m}$, for $m=1, 2, 3, 4$. We will further assume that the metric $g$ coincides with the standard flat metric $g_0 = g_{\rl^{m}} + g_{\tor^{4-m}}$ at infinity in each end. The notation $g_{\tor^{4-m}}$ is for the flat metric obtained as a quotient of $\rl^{4-m}$ by any lattice. The case $m=3$ will include slightly more sophisticated situations, like in \cite{Mi1}. Basically, the Hopf fibration $\pi: \sph^3 \to \sph^2$ can be extended radially into $\pi: \rl^4 \backslash \set{0} \to \rl^3 \backslash \set{0}$ and we may assume that $M \backslash K$ is the total space of (a restriction of) this circle fibration. Then we define the model metric 
at infinity to be $g_0 := \pi^*g_{\rl^{3}} + \eta^2$, where $\eta$ is any constant multiple of the standard contact form on the three-sphere (\cite{Mi1}). 
Note also that all we will say will remain true if $g$ is only asymptotic to $g_0$, thanks to perturbation arguments (cf. \cite{Mi1} for instance). The analysis on such spaces is somehow understood, so we will drop the proofs. The reader interested in the details of this analytical material is referred to \cite{Mel,HHM} for the Mazzeo-Melrose approach or to \cite{MP,Mi1} for softer arguments. 

We will denote by $r$ the Euclidean distance to the origin in $\rl^{m}$. In what follows, we will always write $A_R$ for the ``annulus'' defined by $R\leq r \leq 2R$ and $A^\kappa_R$ for $2^{-\kappa} R \leq r \leq 2^{\kappa +1} R$ ($\kappa \geq 0$). Similarly, the ``balls'' $K \cup \set{r \leq R}$ will be denoted by $B_R$. 

\subsection{The Sobolev theory}

Given a real number $\delta$ and a subset $\Omega$ of $M$, we first define the weighted Lebesgue space
$L^2_\delta (\Omega)$ as the set of functions $u \in L^2_{loc}(\Omega)$ such that the following norm is finite: 
$$
\norm{u}_{L^2_\delta(\Omega)} := \left( \int_{\Omega \cap K} u^2  
+ \int_{\Omega \backslash K} u^2 r^{-2\delta}  \right)^{\frac{1}{2}}.
$$
We will often write $L^2_\delta$ for 
$L^2_\delta (M)$. The following should be kept in mind: 
$$
r^{a} \in L^2_{\delta}(M \backslash K) \Leftrightarrow \delta > \frac{m}{2} + a.
$$

Any function $u$ on $M \backslash K$ can be written
$
u = \Pi_0 u + \Pi_\bot u
$
where $\Pi_0 u$ is obtained by computing the mean value of $u$ along $\tor^{4-m}$. 
In other words, $\Pi_0 u$ is the part in the kernel of the Laplacian on $\tor^{4-m}$ while $\Pi_\bot u$ lies in the positive eigenspaces of this operator. The point is these projector commute with the Laplacian and elliptic estimates will be different for $\Pi_0 u$ and $\Pi_\bot u$. We therefore introduce the Hilbert space $L^2_{\delta,\epsilon} (\Omega)$ of functions $u\in L^2_{loc}(\Omega)$ such that $\norm{\Pi_0 u}_{L^2_\delta(\Omega \backslash K)}$ and $\norm{\Pi_\bot u}_{L^2_\epsilon(\Omega \backslash K)}$ are finite. The good Sobolev space for us is the Hilbert space 
$H^2_{\delta}$ of functions $u \in H^2_{loc}$ such that $\nabla^k \Pi_0 u \in L^2_{\delta-k}$ and $\nabla^k \Pi_\bot u \in L^2_{\delta-2}$ for $k=0,1,2$. 

To state the main a priori estimate, we need a definition. We will say that the exponent $\delta$ is \emph{critical}
if $r^{\delta - \frac{m}{2}}$ is the (pointwise) order of growth of an harmonic function on $\rl^m\backslash \set{0}$. 
More precisely, the critical values correspond to $\delta - 2 \in \ir \backslash \set{-1}$ when $m=4$, $\delta - \frac{3}{2} \in \ir$ when $m=3$, $\delta - 1 \in \ir$ when $m=2$, $\delta -\frac12 = 0$ or $1$ when $m=1$. When $m=2$, the value $\delta=1$ is doubly critical, owing to the constants and  the harmonic function $\log r$. When $m=1$, there are only two critical values because the Laplacian on $\rl$ is also (minus) the Hessian, so that harmonic functions are affine ; in this case, exponential weights are usually used, but we will not really need them and we prefer to give a general framework including faster than linear volume growths. Note also that when one of the ends of $M$ is a non-trivial finite quotient of the model, some critical values (as defined above) may turn irrelevant: for instance, there is no harmonic function with exactly linear growth on $\rl^2/\pm$, which makes $\delta=0$ and $\delta=2$ non-critical.

We are interested in the unbounded operator 
\operator{P_\delta}{\mathcal{D}(P_\delta)}{L^2_{\delta-2,\delta-2}}{u}{\Delta u} 
whose domain $\mathcal{D}(P_\delta)$ is the dense subset of $L^2_{\delta,\delta-2}$ 
whose elements $u$ have their Laplacian in $L^2_{\delta-2,\delta-2}$. The usual $L^2$ pairing 
identifies the topological dual space of $L^2_{\delta,\delta-2}$
(resp. $L^2_{\delta-2,\delta-2}$) with $L^2_{-\delta,2-\delta}$ (resp. $L^2_{2-\delta,2-\delta}$). 
For this identification, the adjoint $P^*_\delta$ of $P_\delta$ is 
\operator{P^*_\delta}{\mathcal{D}(P^*_\delta)}{L^2_{-\delta,2-\delta}}{u}{\Delta u} 
where the domain $\mathcal{D}(P^*_\delta)$ is the dense subset of $L^2_{2-\delta}$ 
whose elements $u$ have their Laplacian in $L^2_{-\delta,2-\delta}$. 
The following proposition can be proved for instance along the lines of Proposition 1 in \cite{Mi1}.

\begin{prop}
If $\delta$ is non-critical, then $P_\delta$ is Fredholm and its cokernel is the kernel of $P_\delta^*$.
\end{prop}

The following property is classical in this context and makes it possible to understand precisely the growth 
of solutions to our equations (cf. Lemma 5 in \cite{Mi1}).

\begin{prop}\label{sauts}
Suppose $\Delta u = f$ with $u$ in $L^2_{\delta}(B_{R_0}^c)$ and $f$ in $L^2_{\delta'-2}(B_{R_0}^c)$ for non-critical exponents $\delta > \delta'$ and a large number $R_0$. Then in each end of $M$, we may write $u= h + v$, 
where $h$ is a harmonic function on $\rl^m \backslash \set{0}$ and $v$ is in $L^2_{\delta',\delta'-2}$.
\end{prop}

For instance, if $m=1$ and $f$ is a smooth and compactly supported function, we obtain that, 
in each end of $M$, $v$ lies in $L^2_\delta$ for every $\delta$. Since $\Delta v = f$, we can use standard elliptic estimates 
such as Lemma \ref{nashmoser} (below) to see that $v = \OO(r^{-a})$ for every $a$ (together with its derivatives, indeed). 
We will abbreviate this by $v = \OO(r^{-\infty})$. So a solution $u$ of $\Delta u = f$ behaves in each end like an affine function 
on $\rl$, up to $\OO(r^{-\infty})$.

This proposition also implies that $P_\delta$ is injective as soon as $\delta- \frac{m}{2}< 0$
and, by duality, surjective as soon as $\delta - \frac{m}{2} > 2 - m$ (cf Corollary 2 in \cite{Mi1}).

Finally, as a by-product of the theory (cf. Lemma 4 in \cite{Mi1}), we are given, for every (large) number $R_0$,
and every non-critical $\delta < \frac{m}{2}$, a bounded operator 
\begin{equation}\label{inv}
\appli{G_{R_0}}{L^2_{\delta-2}(B_{R_0}^c)}{H^2_{0,\delta}(B_{R_0}^c)} 
\end{equation} 
which is an inverse for the Laplacian. Its domain $H^2_{0,\delta}$ is the space of functions
$u \in H^2_{\delta}$ such that $\Pi_\bot u$ vanishes along $\partial B_{R_0}$. On $\ker \Pi_0$,
$G_{R_0}$ is defined by first solving the equation on the domains $B_R \backslash B_{R_0}$ with Dirichlet 
boundary condition and then letting $R$ go to infinity. On $\ker \Pi_\bot$, it is given by an explicit 
formula. For instance, when $m=1$, we set for each $f \in \ker \Pi_\bot$: 
$$
G_{R_0} f := \int_{R_0}^r (\rho-r) f(\rho) d\rho.
$$

\subsection{From integral to pointwise bounds}

In view of handling (weighted) H\"older norms, more adapted to nonlinear analysis, 
the following Moser inequality is useful:
\begin{equation}\label{nashmoser}
\norm{u}_{L^\infty(A_R)} \leq c \left( \frac{1}{\sqrt{R^m}} \norm{u}_{L^2(A^1_R)} 
+ R^2 \norm{\Delta u}_{L^\infty(A^1_R)} \right) 
\end{equation}
A way to obtain this consists in lifting the problem to a square-like domain of size $R$ in $\rl^4$ and 
applying the standard elliptic estimate on $\rl^4$ ; the behaviour of the constants with respect to $R$ follows 
from scaling and counting fundamental domains. As a consequence of this inequality, the inverse $G_{R_0}$ for the Laplacian 
on exterior domains (cf. \ref{inv}) obeys an $L^\infty$ estimate. The proof relies on an idea that can be found in \cite{MP,Biq}.

\begin{lem}\label{estimG}
Given positive numbers $R_0$ and $a$, there is a constant $c = c(R_0,a)$ such that for every continuous function $f$ 
on $B_{R_0^c}$ with $f = \OO(r^{-a})$,
$$
\norm{r^{a+2} G_{R_0} f}_{L^\infty} \leq c \norm{r^a f}_{L^\infty} . 
$$
\end{lem}

\begin{proof}
First, write $f=\Pi_0 f + \Pi_\bot f$ and observe that $\Pi_0 f$ is obtained as an integral along the $\tor^{4-m}$ factor,
so that the sup norms of both terms can be estimated by the sup norms of $f$. We may therefore tackle them separately.
The case $f=\Pi_0 f$ consists in using the explicit formula used to define $G_{R_0}$ on $\ker \Pi_\bot$, so we assume 
$f= \Pi_\bot f$. Then $G_{R_0} f$ vanishes along $\partial B_{R_0}$. Let us put $R_i := 2^i R_0$. Using a partition of unity,
we may write $f= \sum_i f_i$ with $\supp f_i \subset A_{R_i}$ and $\abs{f_i} \leq \abs{f}$. Then \ref{nashmoser} yields:
$$
R_i^a \norm{G_{R_0} f_j}_{L^\infty(A_i)} \leq c \; R_i^{a+2} \norm{f_j}_{L^\infty(A_{R_i}^1)} 
+ c \; R_i^{-\delta_a} \norm{G_{R_0} f_j}_{L^2(A_{R_i}^1)},
$$
where $\delta_a = \frac{m}2-a$ (note that $A_{R_0}^1$ should be understood as $B_{4R_0} \backslash B_{R_0}$ and that the corresponding Moser-type estimate near the boundary is standard). Picking any $\delta$ close to $\delta_a$, we get
\begin{eqnarray*}
R_i^{-\delta_a} \norm{G_{R_0} f_j}_{L^2(A_{R_i}^1)} \leq c \; R_i^{\delta-\delta_a} \norm{G_{R_0} f_j}_{L^2_{\delta}(A_{R_i}^1)}
&\leq& c \; R_i^{\delta-\delta_a} \norm{f_j}_{L^2_{\delta-2}} \\
&\leq& c \; \left( \frac{R_i}{R_j} \right)^{\delta-\delta_a} \norm{r^{a+2} f_j}_{L^\infty}.
\end{eqnarray*} 
Now, given $i$ and $j$, we choose $\delta$ so that $\delta-\delta_a$ is $\epsilon$ times the sign of $j-i$ for some small positive number $\epsilon$ (and zero if $i=j$). Then we find
$$
R_i^a \norm{G_{R_0} f_j}_{L^\infty(A_{R_i})} \leq c \;   2^{-\epsilon \abs{j-i}} \norm{r^{a+2} f}_{L^\infty}. 
$$
Summing over $j$ leads to:
$$
R_i^a \norm{G_{R_0} f}_{L^\infty(A_i)} \leq c \; \norm{r^{a+2} f}_{L^\infty} 
$$
and the result follows at once.
\end{proof}

\end{document}